\documentclass[a4paper, reqno]{amsart}
\newif\iffinal
\finaltrue

\pdfoutput=1 

\usepackage[utf8]{inputenc}
\usepackage[T1]{fontenc}
\usepackage{lmodern}
\usepackage[english]{babel}

\usepackage[shortlabels]{enumitem}

\usepackage{amssymb}
\usepackage{colonequals}
\usepackage{mathtools}
\usepackage{hyperref}
\usepackage{colortbl}
\usepackage{color}
\usepackage{stmaryrd}

\DeclareMathOperator{\Int}{Int}

\newcommand{\N}{\mathbb{N}}

\newcommand{\Z}{\mathbb{Z}}
\newcommand{\val}{v}

\DeclareMathOperator{\fixdiv}{d}

\DeclarePairedDelimiter{\norm}{\lVert}{\rVert}

\makeatletter
\@namedef{subjclassname@2020}{\textup{2020} Mathematics Subject Classification}
\makeatother

\makeatletter
\newtheorem*{rep@theorem}{\rep@title}
\newcommand{\newreptheorem}[2]{%
\newenvironment{rep#1}[1]{%
 \def\rep@title{#2 \ref{##1}}%
 \begin{rep@theorem}}%
 {\end{rep@theorem}}}
\makeatother

\newreptheorem{theorem}{Theorem}

\newtheorem{theorem}{Theorem}
\newtheorem{proposition}{Proposition}[section]

\theoremstyle{definition}
\newtheorem{definition}[proposition]{Definition}

\newtheorem{remark}[proposition]{Remark}

\newtheorem{fact}[proposition]{Fact}

\numberwithin{equation}{section}

\author{Sarah Nakato}
\address{
Department of Mathematics\\ Kabale University\\ Plot 364
  Block 3 Kikungiri Hill\\ Kabale\\ Uganda}
\email{\href{mailto:snakato@kab.ac.ug}{snakato@kab.ac.ug}}

\author{Roswitha Rissner}
\address{Department of Mathematics\\University of Klagenfurt\\
  Universitätsstraße 65-67\\9020 Klagenfurt am Wörthersee\\Austria}
\email{\href{mailto:roswitha.rissner@aau.at}{roswitha.rissner@aau.at}}
\thanks{This research was funded in part by the Austrian Science
  Fund~(FWF)~[10.55776/DOC78]. For open access purposes, the authors
  have applied a CC~BY public copyright license to any author-accepted
  manuscript version arising from this submission.}

\title[{Irreducible integer-valued polynomials}]{Irreducible integer-valued polynomials with prescribed minimal
  power that factors non-uniquely}



\begin{document}

\maketitle

\begin{abstract}
  We study the question up to which power an irreducible
  integer-valued polynomial that is not absolutely irreducible can
  factor uniquely. For example, for integer-valued polynomials over
  principal ideal domains with square-free denominator, already the
  third power has to factor non-uniquely or the element is absolutely
  irreducible.  Recently, it has been shown that for any $N\in\N$,
  there exists a discrete valuation domain $D$ and a polynomial
  $F\in\Int(D)$ such that the minimal $k$ for which $F^k$ factors
  non-uniquely is greater than $N$.

  In this paper, we show that, over principal ideal domains with
  infinitely many maximal ideals of finite index, the minimal power
  for which an irreducible but not absolutely irreducible element has
  to factor non-uniquely depends on the $p$-adic valuations of the
  denominator and cannot be bounded by a constant.


\end{abstract}

\vspace*{0.5cm}
\begin{center}
  \textit{Dedicated to Sophie Frisch on the occasion of her
    60\textsuperscript{th} birthday.}
\end{center}

\section{Introduction}
The building blocks in the study of factorizations of elements in
commutative rings are the irreducible elements. In rings that allow
non-unique factorizations, it is to be expected that even powers of
irreducible elements $c$ factor non-uniquely, that is, $c^k$ may have
a factorization essentially different from $c \cdots c$.

\begin{definition}
  An irreducible element is called \emph{absolutely irreducible} if
  all its powers factor uniquely.
\end{definition}

The notion of absolutely irreducible elements forms a bridge between
irreducible elements and prime elements. A thorough understanding of
the factorization behaviour of a ring entails a comprehensive study of
the (non-)absolutely irreducibles. For rings of number fields, Scott
Chapman and Ulrich
Krause~\cite[Theorem~3.1]{Chapman-Krause:2012:Atomic-decay} gave a
characterization for absolutely irreducible elements. Alfred
Geroldinger and Franz Halter-Koch gave a characterization for reduced
Krull monoids~\cite[Proposition~7.1.4
and~7.1.5]{Geroldinger-HalterKoch:2006:nuf}. In general, recognizing
absolutely irreducible polynomials appears to be a difficult task.
Absolutely irreducible elements have also been called completely
irreducible~\cite{Kaczorowski:1981:completely-irred} and strong
atoms~\cite{Baginski-Kravitz:2010:HFKR,Chapman-Krause:2012:Atomic-decay}. A
closely related notion are \emph{elementary
  atoms}~\cite[Chapter~4]{Grynkiewicz:2022:elasticites}.

In this paper, we study the factorizations of powers of irreducible
elements in rings of integer-valued polynomials, that is,
\begin{equation*}
  \Int(D) = \{F \in K[x] \mid F(D) \subseteq D\}
\end{equation*}
where $D$ is a principal ideal domain with quotient field $K$. These
rings are well-known to contain both absolutely irreducible elements
and irreducible elements that are not absolutely irreducible,
cf.~\cite{Angermueller:2022:strong-atoms,Nakato:2020:NonAbs}.  Our
focus is set on the minimal power of an irreducible element which has
more than one factorization.

Factorization-theoretic studies focused mostly on Krull monoids or, in
the context of rings, Krull domains. The rings of integer-valued
polynomials that we study here are not Krull but
Prüfer~\cite{Cahen-Chabert-Frisch:2000:interpolation,Loper:1998:prüfer}. However,
each monadic submonoid~$\llbracket F \rrbracket$---the divisor-closed
submonoid generated by all powers of $F$---of $\Int(D)$ is Krull
provided that $D$ is
Krull~\cite{Reinhart:2014:monadic,Frisch:2016:monadic}. Hence, our
study of the factorizations of powers of irreducible integer-valued
polynomials takes place in the Krull setting as well. Rings of
integer-valued polynomials have been studied for their
factorization-theoretic behavior in the recent decades,
cf.~\cite{AndersonS-Cahen-Chapman-Smith:1995:fac-iv,Antoniou-Nakato-Rissner:2018:table-crit,Cahen-Chabert:1995:Elasticity-for-IVP,Chapman-McClain:2005:irred-iv-poly,Fadinger-Frisch-Windisch:2023:lengths,Fadinger-Windisch:2023:lengths,Frisch:2013:prescribed-sets,Frisch-Nakato-Rissner:2018:fac,Peruginelli:2015:square-free-denom,Tichy-Windisch:2024:carlitz}.

In recent years there has been progress in characterizing absolutely
irreducible elements in these rings: In the joint
paper~\cite{Rissner-Windisch:2020:binom-polys} with Daniel Windisch,
the second author has verified the decade-long conjecture that the
binomial polynomials are absolutely irreducible in $\Int(\Z)$. In
subsequent collaboration with Sophie Frisch, the authors gave a
characterization of the completely split absolutely irreducible
integer-valued polynomials over a discrete valuation
domain~\cite[Theorem~2]{Frisch-Nakato-Rissner:2022:split}.

The first author gave, again in collaboration with Sophie Frisch, a
graph-theoretic characterization of absolute irreducibility of
integer-valued polynomials on principal ideal domains whose
denominators are
square-free~\cite[Theorem~3]{Frisch-Nakato:2020:graph-theoretic}.
Their proof provides a particularly neat verification mechanism for
absolute irreducibility which we recall at this point.

\begin{fact}[{\cite[Remark~3.2]{Frisch-Nakato:2020:graph-theoretic}}]\label{fact:square-free-bound}
  Let $D$ be a principal ideal domain and let $F\in \Int(D)$ be an
  irreducible polynomial with square-free denominator.

  If $F^3$ has a unique factorization in $\Int(D)$, then $F$ is
  absolutely irreducible.
\end{fact}

Following up on the last two references, Moritz Hiebler and the
authors gave a complete characterization of the absolutely irreducible
integer-valued polynomials over a discrete valuation
domain~\cite[Theorem~2]{Hiebler-Nakato-Rissner:2023:abs-irred}.
Additionally, they explicitly established upper bounds for the minimal
power $S$ that has to factor non-uniquely for irreducible elements
that are not absolutely irreducible. One of these bounds only depends
on the size of the residue field of the base ring and the valuation of
the denominator. This can be considered as a first step towards a
generalization of the square-free case
in~Fact~\ref{fact:square-free-bound}. Moreover, they have shown that
$S$ cannot be bounded by a constant. Indeed, for any $N$ there exists
a discrete valuation domain and a polynomial $F$ such that $F^n$
factors non-uniquely for the first time for an $n\ge N$.

In the publication at hand, we show in Theorem~\ref{theorem:main} that
in every principal ideal domain with infinitely many maximal ideals of
finite index there is no constant bound, too. The minimal power of an
irreducible integer-valued polynomial which factors non-uniquely
depends, hence, on the valuation of the denominator.  Indeed, we show
that for all $N\in \N$ there exists an irreducible polynomial
$F\in \Int(D)$ whose powers $F^n$ factor uniquely when $n < N$, but
$F^N$ does not. Furthermore, we show that $F^N$ has exactly two
essentially different factorizations, one of length $2$ and one of
length $N$.

The proof of this theorem is located in
Section~\ref{section:main-result}. Before, in
Section~\ref{section:preliminaries}, we collect the required
preliminaries.

\section{Preliminaries}
\label{section:preliminaries}
Throughout, let $D$ be a principal ideal domain. For a prime element
$p$ (prime ideal $P$), we denote the $p$-adic ($P$-adic) valuation by
$\val_p$ ($\val_P$).  As usual, for an element $r$ of a commutative
ring, a \emph{factorization} is an expression of $r$ as the product
$c_1\cdots c_k$ of irreducible elements. The number $k$ of irreducible
factors is called the \emph{length} of this factorization. Two
factorizations of an element are \emph{essentially the same} if the factors are
equal up to permutation and multiplication by units.

For a polynomial $F\in \Int(D)$, the \emph{fixed divisor} is
defined as the ideal
\begin{equation*}
  \fixdiv(F)=  (F(a) \mid a\in D).
\end{equation*}
If $\fixdiv(F) = (d)$ we call, by abuse of notation, $d$ the fixed
divisor of $F$.  We call a polynomial \emph{image-primitive} if its
fixed divisor is equal~$1$. Note that an irreducible integer-valued
polynomial is necessarily image-primitive.

For $f\in D[x]$ and $0\neq b\in D$, the polynomial $\frac{f}{b}$ is
an element of $\Int(D)$ if and only if $b \mid \fixdiv(f)$. Moreover,
for any prime element $p\in D$, the following implication holds:
\begin{equation}\label{eq:primes-with-large-index}
  p\mid \fixdiv(f) \quad \Longrightarrow\quad f\in pD[x] \quad\text{ or }\quad \norm{pD} = |D/pD| \le \deg(f).
\end{equation}

For a thorough introduction into factorizations and integer-valued
polynomials we refer to \cite{Geroldinger-HalterKoch:2006:nuf}
and~\cite{Cahen-Chabert:1997:IVP,Cahen-Chabert:2016:survey}, respectively.

We close this section with two results that we need for the
construction below in the proof of Theorem~\ref{theorem:main}.

\begin{fact}[{\cite[Lemma~3.3]{Frisch-Nakato-Rissner:2018:fac},
    \cite[Lemma~6.2]{Hiebler-Nakato-Rissner:2023:abs-irred}}]
  \label{fact:irreducible-replacements}
  Let $D$ be a Dedekind domain with infinitely many maximal ideals and
  $K$ its quotient field. Further, let $g_1$, \dots, $g_q\in D[x]$ be
  monic, non-constant polynomials, and set
  $d = \sum_{i=1}^q \deg(g_i)$.

  For every $m\in \N_0$, there exist monic polynomials $f_1$, \dots,
  $f_q$ which satisfy
  the following properties:
  \begin{enumerate}
  \item $\deg(f_i) = \deg(g_i)$,
  \item $f_1$, \dots, $f_q$ are pairwise non-associated irreducible
    polynomials in $K[x]$, and
  \item $f_i \equiv g_i \mod P^{m+1}D[x]$ for all prime ideals $P$ of
    $D$ with $\norm{P} \le d$.
  \end{enumerate}
  In particular, if $P$ is a prime ideal of $D$ with $\norm{P} \le d$
  and $a\in D$ then
  \begin{equation*}
    \val_P(g_i(a)) \le m \text{ or } \val_P(f_i(a)) \le m \quad \Longrightarrow \quad \val_P\!\left(f_i(a) \right) = \val_P\!\left(g_i(a) \right)
  \end{equation*}
\end{fact}

\begin{fact}[{Special case
    of~\cite[Lemma~2.5]{Frisch-Nakato:2020:graph-theoretic}}]
  \label{fact:quintessential}
  Let $D$ be a principal ideal domain, $p\in D$ a prime element, and
  $f_1$, \dots, $f_q\in D[x]$ be monic, non-constant, irreducible
  polynomials such that
  \begin{enumerate}
  \item $\fixdiv\!\left(\prod_{i=1}^qf_i\right) = p^e$ and
  \item for $1 \leq j< q$, there exists $w_j \in D$ with
    \begin{equation*}
      \val_{p}(f_i(w_j)) =
      \begin{cases}
        e & 1 \le i = j < q, \\
        0 & i\neq j.
      \end{cases}
    \end{equation*}
  \end{enumerate}

  Then for each $n\in \N$, every factorization of
  $\left(\frac{\prod^k_{i=1}f_i}{p^e}\right)^{\!n}$ in $\Int(D)$ into
  not necessarily irreducible factors is essentially of the form
  \begin{equation*}
    \frac{\left(\prod^{q-1}_{i=1}f_i\right)^{\!a_1}f_q^{b_1}}{p^{ea_1}} \cdot
    \frac{\left(\prod^{q-1}_{i=1}f_i\right)^{\!a_2}f_q^{b_2}}{p^{ea_2}}
  \end{equation*}
  with $a_1 + a_2 = b_1+b_2 = n$.
\end{fact}

\begin{remark}
  In the language of \cite{Frisch-Nakato:2020:graph-theoretic}, the
  existence of the $w_j$ in Fact~\ref{fact:quintessential} is referred
  to as $f_j$ being \emph{quintessential} for $p$
  among the $f_1$, \dots, $f_q$.
\end{remark}

\section{Main result}
\label{section:main-result}
\begin{theorem} \label{theorem:main} Let $D$ be a principal ideal
  domain with infinitely many prime ideals of finite index.  For each
  integer $N\ge 2$ there exists an irreducible element $F\in \Int(D)$
  such that
  \begin{enumerate}[label=(\alph*)]
  \item\label{theorem:main:1} $F^n$ factors uniquely for all $n< N$ and
  \item\label{theorem:main:2} $F^N$ has exactly two essentially different
    factorizations, one of length $2$ and one of length $N$.
  \end{enumerate}
\end{theorem}

\begin{proof}
  \textbf{Step 1:} We construct a suitable polynomial $F$.

  Let $P = (p)$ be a prime ideal of $D$ and $q=\norm{P}$.  Further,
  let $a_1$, $a_2$, \dots, $a_{q}$ be a complete system of residues
  modulo $p$ which satisfies $a_i \equiv 0 \mod Q$ for all prime
  ideals $Q\neq P$ with $\norm{Q} \le q$.  We define
  \begin{equation*}
    g_i =
    \begin{cases}
      (x-a_i)^{N-1} & 1 \leq i < q, \\
      (x-a_q)^{N}   & i = q.
    \end{cases}
  \end{equation*}
  A straight-forward verification shows that
  \begin{equation}\label{eq:val>=}
    \val_P(g_i(c)) \ge
    \begin{cases}
      N-1 & \val_P(c- a_i) \ge 1 \text{ and } 1\le i < q,\\
      N   & \val_P(c- a_q) \ge 1 \text{ and } i=q \\
    \end{cases}
  \end{equation}
  and
  \begin{equation}\label{eq:val=}
    \val_P(g_i(c)) =
    \begin{cases}
      N-1 & \val_P(c- a_i) = 1 \text{ and } 1\le i < q,\\
      N   & \val_P(c- a_q) = 1 \text{ and } i=q, \\
      0   & \val_P(c- a_i) = 0 \text{ and } 1\le i \le q\\
    \end{cases}
  \end{equation}
  holds. Moreover, due to the choice of the $a_i$ and the property
  in~\eqref{eq:primes-with-large-index}, for all prime ideals
  $Q \neq P$ of $D$, there exists an element $z$ of $D$ such that
  \begin{equation}\label{eq:otherprimes}
    \val_Q\!\left(\prod^q_{i=1}g_i(z)\right) = 0.
  \end{equation}
  It follows that
  \begin{equation}
    \label{eq:y}
    \fixdiv\!\left(\prod^q_{i=1}g_i\right) = p^{N-1}.
  \end{equation}

  We apply Fact~\ref{fact:irreducible-replacements} for $m=N$: Let
  $f_1$, $f_2$, \dots, $f_q$ be the irreducible polynomials.
  Fact~\ref{fact:irreducible-replacements} ensures that
  $\val_Q(f_i(a)) = \val_Q(g_i(a))$ holds for all $a\in D$ and all
  prime ideals $Q$ with $\norm{Q}\le \sum_{i=1}^q \deg(f_i)$ with
  $\val_Q(g_i(a)) \le N$ or $\val_Q(f_i(a)) \le N$.  In particular,
  $f_1$, \dots, $f_q$ satisfy the same (in)equalities as the $g_i$,
  that is, \eqref{eq:val>=},~\eqref{eq:val=},
  and~\eqref{eq:otherprimes} hold when we replace the $g_i$ by the
  $f_i$ (where for prime ideals $Q$ that are not covered by
  Fact~\ref{fact:irreducible-replacements} the latter assertion holds
  due to~\eqref{eq:primes-with-large-index}).  Hence
  \begin{equation*}
    \fixdiv\!\left(\prod^q_{i=1}f_i\right) =\fixdiv\!\left(\prod^q_{i=1}g_i\right) = p^{N-1}.
  \end{equation*}
  We set
  \begin{equation*}
    F = \frac{\prod^q_{i=1}f_i}{p^{N-1}}
  \end{equation*}
  and note that $F$ is an image-primitive element of $\Int(D)$.

  \textbf{Step 2:} We verify that $F$ is irreducible and satisfies the
  assertions~\textit{\ref{theorem:main:1}}
  and~\textit{\ref{theorem:main:2}}.  For each $1 \leq j \le q$, let
  $c_j \in D$ be elements with $\val_P(c_j - a_j) = 1$.  It follows
  from Fact~\ref{fact:quintessential} that, for $n>1$, every
  factorization of $F^n$ (into not necessarily irreducible elements)
  is essentially of the form
  \begin{equation}\label{eq:factorizations}
    F^n = \frac{\left(\prod^{q-1}_{i=1}f_i\right)^{\!a_1}f_q^{b_1}}{p^{a_1(N-1)}} \cdot
    \frac{\left(\prod^{q-1}_{i=1}f_i\right)^{\!a_2}f_q^{b_2}}{p^{a_2(N-1)}}
  \end{equation}
  where $a_1$, $a_2$, $b_1$, $b_2\ge 0$ are integers with
  \begin{equation}
    \label{eq:ab}
    a_1 + a_2 = b_1 + b_2 = n.
  \end{equation}
  We assume without restriction that $b_1 \le a_1$.  As both factors
  in \eqref{eq:factorizations} are integer-valued, the equations
  in~\eqref{eq:val=} in combination with
  Fact~\ref{fact:irreducible-replacements} imply that
  \begin{align*}\label{eq:ivineq}
    \begin{split}
      b_1N
      &=  \val_P\!\left(g_q^{b_1}(c_q)\right) \\
      &= \val_P\!\left(\prod^{q-1}_{i=1}g_i^{a_1}(c_q)g_q^{b_1}(c_q)\right) \\
      &= \val_P\!\left(\prod^{q-1}_{i=1}f_i^{a_1}(c_q)f_q^{b_1}(c_q)\right) \\
      &\ge a_1(N-1).
    \end{split}
  \end{align*}
  This further implies that
  \begin{equation*}
    0 \le b_1N - a_1 (N-1) = (b_1-a_1)N + a_1
  \end{equation*}
  or, equivalently,
  \begin{equation}\label{eq:ineq1}
    (a_1-b_1)N \le a_1.
  \end{equation}
  The latter then yields
  \begin{equation}\label{eq:ineq2}
     (a_1-b_1)N \leq a_1 + a_2 = n.
  \end{equation}
  Since $b_1 \le a_1$ by assumption, it follows from
  Equation~\eqref{eq:ineq2} that, whenever $n<N$, then $a_1 = b_1$
  must hold. In this case, also $a_2 = b_2$ follows from
  Equation~\eqref{eq:ab}, and Equation~\eqref{eq:factorizations}
  reduces to
  \begin{equation}
    F^n
    = \left(\frac{\left(\prod^{q-1}_{i=1}f_i\right)f_q}{p^{N-1}}\right)^{\!a_1} \cdot
    \left(\frac{\left(\prod^{q-1}_{i=1}f_i\right)f_q}{p^{N-1}}\right)^{\!a_2}
    = F^{a_1} F^{a_2}.
  \end{equation}
  In summary, if $n<N$ then $F^n$ factors uniquely.  This completes
  the proof of~\textit{\ref{theorem:main:1}}.

  Note that the argument above in particular applies to the case
  $n=1$. Hence, in this case, either $a_1=b_1=0$ or $a_2=b_2=0$ and
  $F$ is irreducible.

  To prove \textit{\ref{theorem:main:2}}, we set $n=N$ in
  Equation~\eqref{eq:factorizations}. By~Equation~\eqref{eq:ineq2} it
  follows that $(a_1-b_1)N \le N$ which implies that either
  $a_1 = b_1$ or $a_1=b_1+1$. As above, the first case leads to the
  trivial factorization $F^{N} = F \cdots F$ which is of
  length~$N$.

  From now on we assume that $a_1 = b_1+1$. It follows
  from~\eqref{eq:ab} and~\eqref{eq:ineq1} that
  \begin{equation}
    N \le a_1 = b_1 + 1 \le N.
  \end{equation}
  which immediately implies that $a_1= N$, $a_2 = 0$, $b_1= N-1$ and
  $b_2 = 1$. Since the $f_i$ satisfy the analogous versions
  of~\eqref{eq:val>=} and~\eqref{eq:val=} by
  Fact~\ref{fact:irreducible-replacements}, it follows that for all
  $c\in D$,
  \begin{equation*}
    \val_P\!\left(\left(\prod^{q-1}_{i=1}f_i^{N}(c)\right)f_q^{N-1}(c)\right)
    = \sum_{i=1}^{q-1}N \val_P(f_i(c)) + (N-1) \val_P(f_q(c))
    \ge N(N-1)
  \end{equation*}
  which further implies that
  \begin{equation*}
    F^{N} = \frac{\left(\prod^{q-1}_{i=1}f_i^{N}\right)f_q^{N-1}}{p^{N(N-1)}} \cdot  f_q
  \end{equation*}
  is a factorization of $F^N$ into irreducibles of length $2$. We have
  found all essentially different factorizations of $F^N$.
\end{proof}

%
\bibliographystyle{plain}
\bibliography{bibliography}

\end{document}